\documentclass[preprint,12pt]{amsart}

\usepackage[margin=1.35in]{geometry}
\usepackage{pdftexcmds}

\usepackage{amsmath,calligra}
\usepackage{amsfonts}
\usepackage{amssymb}
\usepackage{xypic}
\usepackage{longtable}
\usepackage{amsthm}
\usepackage{booktabs}
\usepackage{caption}
\usepackage{mathrsfs}

\DeclareMathOperator{\sheafhom}{\mathscr{H}\text{\kern -3pt {\calligra\large om}}\,}

\theoremstyle{plain}
\newtheorem{theorem}{Theorem}[section]
\newtheorem{corollary}[theorem]{Corollary}
\newtheorem{lemma}[theorem]{Lemma}
\newtheorem{proposition}[theorem]{Proposition}
\theoremstyle{definition}

\newtheorem{rem}[theorem]{Remark}

\makeatletter
\def\ps@pprintTitle{%
  \let\@oddhead\@empty
  \let\@evenhead\@empty
  \let\@oddfoot\@empty
  \let\@evenfoot\@oddfoot
}
\makeatother

\title{Galois subspaces for projective varieties}
\author{Robert Auffarth}
\address{R. Auffarth \\Departamento de Matem\'aticas, Facultad de
Ciencias, Universidad de Chile, Santiago\\Chile}
\email{rfauffar@uchile.cl}

\subjclass[2010]{}%

\DeclareFontFamily{U}{cbgreek}{}
\DeclareFontShape{U}{cbgreek}{m}{n}{
        <-6>    grmn0500
        <6-7>   grmn0600
        <7-8>   grmn0700
        <8-9>   grmn0800
        <9-10>  grmn0900
        <10-12> grmn1000
        <12-17> grmn1200
        <17->   grmn1728
      }{}
\DeclareFontShape{U}{cbgreek}{bx}{n}{
        <-6>    grxn0500
        <6-7>   grxn0600
        <7-8>   grxn0700
        <8-9>   grxn0800
        <9-10>  grxn0900
        <10-12> grxn1000
        <12-17> grxn1200
        <17->   grxn1728
      }{}

\makeatletter
\newcommand{\normalorbold}{%
  \ifnum\pdf@strcmp{\math@version}{bold}=\z@ bx\else m\fi
}
\makeatother

\begin{document}

\maketitle

\begin{abstract}
Given an embedding of a projective variety into projective space, we study the structure of the space of all linear projections that, when composed with the embedding, give a Galois morphism from the variety to a projective space of the same dimension.
\end{abstract}

\section{Introduction}
Let $X$ be an irreducible projective variety of dimension $n$ over an algebraically closed field $k$ of characteristic zero, and let $\varphi:X\hookrightarrow\mathbb{P}^N$ be a non-degenerate embedding. If $W\in\mathbb{G}(N-n-1,N)$ is a linear subvariety of $\mathbb{P}^{N}$ of dimension $N-n-1$, we denote by $\pi_W:\mathbb{P}^N\dashrightarrow\mathbb{P}^n$ the linear projection induced by $W$ (which is well-defined modulo automorphisms of $\mathbb{P}^n)$. Following Yoshihara \cite{Yoshi}, we wish to study the set
\begin{equation}
\nonumber\mathbf{G}_{X,\varphi} :=\{W\in\mathbb{G}(N-n-1,N):\pi_W\circ\varphi\text{ is Galois}\},
\end{equation}
where we say that $\pi_W\circ\varphi$ is \textit{Galois} if the field extension $k(\mathbb{P}^n)\hookrightarrow k(X)$ induced by $\pi_W\circ\varphi$ is a finite Galois extension. We will denote by $\mathrm{Gal}(W)$ the Galois group of such an extension. Since $\mathrm{Gal}(W)$ acts on the function field of $X$, we can therefore identify $\mathrm{Gal}(W)$ as a finite subgroup of birational transformations of $X$. Since it is usually more complicated to work with birational transformations than automorphisms, we define the following sets:
\begin{eqnarray}
\nonumber\mathbf{G}_{X,\varphi}^{\mathrm{aut}}&:=&\{W\in\mathbf{G}_{X,\varphi}:\mathrm{Gal}(W)\leq\mathrm{Aut}(X), X/\mathrm{Gal}(W)\simeq\mathbb{P}^n\}\\
\nonumber\mathbf{G}_{X,\varphi}^\varnothing&:=&\{W\in\mathbf{G}_{X,\varphi}:W\cap\varphi(X)=\varnothing\}.
\end{eqnarray}

These sets, and especially $\mathbf{G}_{X,\varphi}^\varnothing$, have been extensively studied under the context of \textit{Galois embeddings}. Indeed, Yoshihara \cite{Yoshi} defines $\varphi$ to be a Galois embedding if $\mathbf{G}_{X,\varphi}^\varnothing\neq\varnothing$, and then proceeds to study Galois groups associated to such linear subspaces. In particular, he proves \cite[Representation 1]{Yoshi} that
\[\mathbf{G}_{X,\varphi}^\varnothing\subseteq\mathbf{G}_{X,\varphi}^{\mathrm{aut}}.\]
There have been several beautiful results established about $\mathbf{G}_{X,\varphi}^{\mathrm{aut}}$ during the past two decades, and in some cases this space even characterizes the pair $(X,\varphi)$ up to automorphisms of $\mathbb{P}^N$. For example, if $X$ is embedded as a normal hypersurface of dimension $n$ into $\mathbb{P}^{n+1}$, Yoshihara \cite{Yoshi3} (in the smooth case) and Fukasawa and Takahashi \cite{FT} (in the normal case) show that $\mathbf{G}_{X,\varphi}=\mathbf{G}_{X,\varphi}^{\mathrm{aut}}$, this space is finite, and they moreover find bounds for its cardinality and characterize those hypersurfaces that achieve the bounds they discover. For example, they prove that if the hypersurface is smooth, then $\mathbf{G}_{X,\varphi}\smallsetminus\mathbf{G}_{X,\varphi}^{\mathrm{aut}}$ is of cardinality at most $\lfloor n/2\rfloor+1$ and $\mathbf{G}_{X,\varphi}^\varnothing$ is of cardinality at most $n+2$, the latter bound being achieved solely in the case of a Fermat hypersurface. Many other similar characterizations have been established, and we recommend the interested reader search out the vast existing literature (for example \cite{Fuk}, \cite{Tak}, \cite{KLT} just to name a few). We remark that Yoshihara and Fukasawa maintain a webpage \cite{YFwebpage} with a very thorough list of open problems related to this theory, as well as an extensive list of references. In general, the strategy has usually been to study $\mathbf{G}_{X,\varphi}^{\mathrm{aut}}$ when it is finite, and understand up to what point its cardinality characterizes the pair $(X,\varphi)$.\\

On the other hand, there are instances where $\mathbf{G}_{X,\varphi}^{\mathrm{aut}}$ (or even $\mathbf{G}_{X,\varphi}^\varnothing$) is not finite. For example, in the case of smooth projective curves (where we note that $\mathbf{G}_{X,\varphi}^\mathrm{aut}=\mathbf{G}_{X,\varphi}$), the author and S. Rahausen study in \cite{AR} the structure of $\mathbf{G}_{X,\varphi}$ for $X=\mathbb{P}^1$ and $\varphi$ a Veronese embedding. The set $\mathbf{G}_{X,\varphi}$, as well as $\mathbf{G}_{X,\varphi}^\varnothing$, is explicitly described as a finite union of families of varying dimensions. In particular, it is shown that $\mathbf{G}_{X,\varphi}^\varnothing$ is infinite (although it does have the structure of quasiprojective variety). In \cite{AuffRah}, the same authors study $\mathbf{G}_{X,\varphi}$ for smooth projective curves of genus $\geq1$, and give an explicit description for the genus 1 case. In the latter article, it is shown that for $g\geq2$, the locus $\mathbf{G}_{X,\varphi}$ inside the Grassmann variety is a disjoint union of projective spaces. For the case $g=1$, it is shown that $\mathbf{G}_{X,\varphi}$ is the disjoint union of points and projective bundles over \'etale quotients of the same elliptic curve, and the number and dimensions of the components of $\mathbf{G}_{X,\varphi}$ are explicitly given in the genus 1 case when $\varphi$ is the embedding associated to a very ample complete linear system.  

The purpose of this article is to give a general structure theorem for  $\mathbf{G}^\mathrm{aut}_{X,\varphi}$ in the case of arbitrary irreducible projective varieties, thereby shedding substantial light on the problem of understanding Galois subspaces. The general idea behind our description of $\mathbf{G}_{X,\varphi}^{\mathrm{aut}}$ is the following:

\vspace{0.2cm}

\begin{enumerate}
\item If we start by fixing a finite group $G\leq\mathrm{Aut}(X)$ such that $X/G\simeq\mathbb{P}^n$, then we can explicitly describe the structure of the subset of $\mathbf{G}_{X,\varphi}^{\mathrm{aut}}$ that consists of all Galois subspaces whose Galois group is \textit{equal} to $G$. This turns out to be a projective space lying in $\mathbb{G}(N-n-1,N)$ that we denote by $\mathbb{P}_{G,\varphi}$ (see Proposition \ref{Galoisgroup}).\\

    \item Next, we can deform $G$ inside $\mathrm{Aut}(X)$ by conjugating it by $\mathrm{Aut}^0(X)$, the connected component of the automorphism group scheme of $X$ that contains the identity. It turns out that all these $\mathbb{P}_{\theta G\theta^{-1},\varphi}$ can be put together into a family (in a loose sense of the term) which can be described as a certain relative proj over a quotient of $\mathrm{Aut}^0(X)$ (see Section \ref{families}).
\end{enumerate}

\vspace{0.2cm}

Given this intuition, our main theorem says the following:

\begin{theorem}\label{main}
Let $X$ be an irreducible projective variety of dimension $n$, let $\varphi:X\hookrightarrow\mathbb{P}^N$ be a non-degenerate embedding, and let $\mathrm{Aut}^0(X)$ denote the irreducible component of the automorphism group scheme of $X$ that contains the identity. Then $\mathbf{G}_{X,\varphi}^{\mathrm{aut}}$ can be described as the disjoint union of relative proj over quotients of $\mathrm{Aut}^0(X)$. Specifically, for each finite subgroup $G\leq\mathrm{Aut}(X)$ such that $X/G\simeq\mathbb{P}^n$, if $N(G)$ denotes the normalizer of $G$ in $\mathrm{Aut}^0(X)$, there exists a coherent sheaf $\mathcal{E}_G$ on the scheme $\mathrm{Aut}^0(X)/N(G)$ and an embedding $\Phi_G:\mathbf{P}(\mathcal{E}_G)\hookrightarrow\mathbf{G}(N-n-1,N)$
whose image lies in $\mathbf{G}_{X,\varphi}^{\mathrm{aut}}$, such that
\[\mathbf{G}_{X,\varphi}^{\mathrm{aut}}=\bigcup_{G}\Phi_G(\mathbf{P}(\mathcal{E}_G)),\]
where the union runs over all subgroups as above. Moreover:
\begin{enumerate}
\item $\Phi_G(\mathbf{P}(\mathcal{E}_G))\cap\Phi_H(\mathbf{P}(\mathcal{E}_H))\neq\varnothing$ if and only if $G$ and $H$ are conjugate by an element of $\mathrm{Aut}^0(X)$, in which case $\Phi_G(\mathbf{P}(\mathcal{E}_G))=\Phi_H(\mathbf{P}(\mathcal{E}_H))$,
\item the fibers of $\mathbf{P}(\mathcal{E}_G)$ over $\mathrm{Aut}^0(X)/N(G)$ are either empty or projective spaces.\\
\end{enumerate}
\end{theorem}

Here we use the notation $\mathbf{P}(\mathcal{E}_G):=\mathrm{Proj}(\mathrm{Sym}(\mathcal{E}_G))$. Note that since $\mathcal{E}_G$ is coherent and not necessarily locally free, $\mathbf{P}(\mathcal{E}_G)$ is not necessarily irreducible, and indeed there are examples where it is not irreducible (see Remark \ref{not irreducible}).

Note that each $\mathbf{P}(\mathcal{E}_G)$ is a family (in a loose sense of the term) of projective spaces over a quotient of $\mathrm{Aut}^0(X)$. This result therefore clarifies the results obtained for curves in genus greater than or equal to 1 in \cite{AuffRah}, since if $X$ is a genus 1 curve, then $\mathrm{Aut}^0(X)=X$, and if it is of genus greater than 1, then $\mathrm{Aut}^0(X)=\{1\}$. As an immediate corollary we obtain the following:

\begin{corollary}
If $\mathrm{Aut}(X)$ is finite, then $\mathbf{G}_{X,\varphi}^{\mathrm{aut}}$ is a finite union of projective spaces.
\end{corollary}

The structure of this article is as follows: In Section \ref{linear} we prove a general result on linear projections between two relative proj over a fixed noetherian scheme. This result could be applied, in the future, to more general versions of Galois subspaces and Galois embeddings (for example in a relative setting, or for general Noetherian schemes over an arbitrary base). In Section \ref{galois} we describe the set of all Galois subspaces with a fixed Galois group, and in Section \ref{families} we deform the group and proceed to prove the structure theorem mentioned above. Section \ref{intermediate} shows an interesting way of seeing Galois subspaces and clarifies a previous strategy used to understand Galois subspaces in a previous article. In Section \ref{questions} we show two examples and ask several questions for future research.

\vspace{0.5cm}

\noindent\textit{Acknowledgements:} The author was partially supported by ANID-Fondecyt Grant 1220997.

\section{A general result on linear projections}\label{linear}

In this section we will prove a general result on linear projections between relative proj. The purpose of establishing this result in such generality is that it can be readily applied to the context we are in, but it can also be used to study much more general contexts in the future. 

Let $S$ be a fixed noetherian scheme, let $X$ be an integral noetherian scheme over $S$ with structural morphism $f:X\to S$, let $\mathcal{A}_1$ and $\mathcal{A}_2$ be graded quasi-coherent $\mathcal{O}_S$-algebras, and let $\varphi_i:X\to\mathbf{Proj}(\mathcal{A}_i)=:Y_i$ be $S$-morphisms for $i=1,2$. We recall that if $\xi:\mathcal{A}_2\to\mathcal{A}_1$ is a morphism of graded $\mathcal{O}_S$-algebras, then by \cite[Section 3.5.1]{EGAII}, there exists a (possibly empty) open subscheme $\mathbf{G}(\xi)\subseteq Y_1$ and an induced morphism 
\[\mathbf{P}(\xi):\mathbf{G}(\xi)\to Y_2\]
that we will call the \textit{rational linear map} associated to $\xi$. If we do not wish to specify the open subscheme $\mathbf{G}(\xi)$ we will simply write $\mathbf{P}(\xi):Y_1\dashrightarrow Y_2$.

We will assume that $\mathcal{A}_1$ and $\mathcal{A}_2$ satisfy the following property that we will denote by $(\star)$:
\begin{enumerate}
    \item The degree 0 part is isomorphic to $\mathcal{O}_S$
    \item The degree 1 part is a coherent $\mathcal{O}_S$-module
    \item The whole sheaf is locally generated by the degree 1 part as an $\mathcal{O}_S$-algebra.
\end{enumerate}

If $\mathcal{B}$ is a graded quasi-coherent $\mathcal{O}_S$-algebra that satisfies $(\star)$, then $\mathbf{Proj}(\mathcal{B})$ represents the functor that associates to $f:T\to S$ the set of pairs $(\mathcal{L},\psi)$ modulo $\sim$, where $\mathcal{L}$ is an invertible $\mathcal{O}_T$-module,
\[\psi:f^*\mathcal{B}\to\bigoplus_{n\geq 0}\mathcal{L}^{\otimes n}\]
is a graded $\mathcal{O}_T$-algebra homomorphism such that $f^*\mathcal{B}_1\to\mathcal{L}$ is surjective and $(\mathcal{L}_1,\psi_1)\sim(\mathcal{L}_2,\psi_2)$ if and only if there is an isomorphism $\beta:\mathcal{L}_2\to\mathcal{L}_1$ such that $\beta\psi_2=\psi_1$. Note that the morphism $T\to\mathbf{Proj}(\mathcal{B})$ induced by the above data is just the composition of
\[\mathbf{P}(\psi):X=\mathbf{Proj}\left(\bigoplus_{n\geq0}\mathcal{L}^{\otimes n}\right)\to\mathbf{Proj}(f^*\mathcal{B})=\mathbf{Proj}(\mathcal{B})\times_SX\]
with the first projection.

Assume that $\varphi_i:X\to Y_i$ is induced by a pair $(\mathcal{L}_i,\psi_i)$. The main result of this section is the following:

\begin{theorem}\label{general theorem}
A morphism of graded $\mathcal{O}_S$-algebras $\xi:\mathcal{A}_2\to\mathcal{A}_1$ satisfies $\mathbf{P}(\xi)\varphi_1=\varphi_2$ on $\varphi_1^{-1}(\mathbf{G}(\xi))$ if and only if there exists $s\in\mathrm{Hom}_{\mathcal{O}_X}(\mathcal{L}_2,\mathcal{L}_1)$ such that the following diagram is commutative:
\begin{equation}\nonumber\xymatrix{(f^*\mathcal{A}_2)_1\ar[d]_{\psi_2}\ar[rr]^{f^*\xi}&&(f^*\mathcal{A}_1)_1\ar[d]^{\psi_1}\\\mathcal{L}_2\ar[rr]_{s}&&\mathcal{L}_1}.\end{equation}
\end{theorem}

The rest of this section will be spent proving this result. 

Let $\xi:\mathcal{A}_2\to\mathcal{A}_1$ be a morphism of graded $\mathcal{O}_S$-algebras such that $\mathbf{P}(\xi)\varphi_1=\varphi_2$ on $\varphi_1^{-1}(\mathbf{G}(\xi))$, and let $\pi_i:Y_i\to S$ be the structure morphism. By \cite[Lemma A.2]{Kleiman}, there exist morphisms of graded $\mathcal{O}_{Y_1}$-algebras 
\[\alpha:\pi_1^*\mathcal{A}_1\to \mathrm{Sym}(\mathcal{O}_{Y_1}(1))\] 
\[\eta:\mathbf{P}(\xi)|_{\mathbf{G}(\xi)}^*\mathrm{Sym}(\mathcal{O}_{Y_2}(1))\to\mathrm{Sym}(\mathcal{O}_{Y_1}(1))|_{\mathbf{G}(\xi)}\]
such that the diagram
\[\xymatrix{\pi_1^*\mathcal{A}_2|_{\mathbf{G}(\xi)}\ar[d]_{\mathbf{P}(\xi)^*\alpha}\ar[rr]^{\pi_1^*\xi}&&\pi_1^*\mathcal{A}_1|_{\mathbf{G}(\xi)}\ar[d]^\alpha\\
\mathbf{P}(\xi)|_{\mathbf{G}(\xi)}^*\mathrm{Sym}(\mathcal{O}_{Y_2}(1))\ar[rr]^\eta&&\mathrm{Sym}(\mathcal{O}_{Y_1}(1))|_{\mathbf{G}(\xi)}}\]
is commutative. By \cite[Lemma A.3]{Kleiman}, we have that the morphism $\mathbf{P}(\alpha)$ associated to $\alpha$ can be identified with the diagonal morphism $\mathbf{P}(\alpha):Y_1\to Y_1\times_S Y_1$.

\begin{lemma}
The morphism $\mathbf{P}(\xi)|_{\mathbf{G}(\xi)}:\mathbf{G}(\xi)\to Y_2$ is given by the pair $(\mathcal{O}_{Y_1}(1)|_{\mathbf{G}(\xi)},\alpha\pi_1^*\xi)$. 
\end{lemma}
\begin{proof}
This is a simple exercise of putting everything together. Indeed, we have that $\mathbf{P}(\pi_1^*\xi)|_{\mathbf{G}(\xi)}$ is the natural morphism
\[\mathbf{P}(\xi)\times\mathrm{id}:\mathbf{G}(\xi)\times_SY_1\to Y_2\times_S Y_2,\]
and by composing $\mathbf{P}(\pi_1^*\xi)\mathbf{P}(\alpha)|_{\mathbf{G}(\xi)}$ with the first projection, we get what we are looking for.
\end{proof}

Define $U:=\varphi_1^{-1}(\mathbf{G}(\xi))$. By pulling back $\alpha\pi_1^*(\xi)$ by $\varphi_1$ and noting that $\varphi_1^*\alpha=\psi_1$, we have that on $U$ the composition $\mathbf{P}(\xi)\varphi_1$ is given by the morphism
\[\varphi_1^*(\alpha\pi_1^*\xi)=\psi_1(f^*\xi):f^*\mathcal{A}_2|_U\to\bigoplus_{n\geq0}\mathcal{L}_1|_U^{\otimes n}.\]
Since on $U$ we have that $\mathbf{P}(\xi)\varphi_1=\varphi_2$, we have that there exists an isomorphism $\beta:\mathcal{L}_2|_U\to\mathcal{L}_1|_U$ such that the diagram
\begin{equation}\nonumber\xymatrix{f^*\mathcal{A}_2\ar[d]_{\psi_2}\ar[rr]^{f^*\xi}&&f^*\mathcal{A}_1\ar[d]^{\psi_1}\\\bigoplus_{n\geq0}\mathcal{L}_2|_U^{\otimes n}\ar[rr]_{\oplus_n\beta^n}&&\bigoplus_{n\geq0}\mathcal{L}_1|_U^{\otimes n}}\end{equation}
commutes. Note that by our hypotheses, it is enough to consider the degree 1 part of the previous diagram.

\begin{lemma}
The isomorphism $\beta$ can be extended to a sheaf homomorphism $\mathcal{L}_2\to\mathcal{L}_1$ that we will denote by $s$.
\end{lemma}
\begin{proof}
We only have to show that $\ker(\psi_2)_1\subseteq\ker(\psi_1(f^*\xi))$, where the subscript 1 indicates the graded 1 piece of $\psi_2$. Let $W\subseteq X$ be open, let $s\in\ker(\psi_2)_1$ and set $t:=\psi_1(f^*\xi)(s)$. Then 
\[t|_{U\cap W}=\beta\psi_2(s|_{U\cap W})=0.\]
This implies that $t|_U=0$, and since $X$ is integral, we obtain that $t=0$.
\end{proof}

Summing up, if $\xi:\mathcal{A}_2\to\mathcal{A}_1$ is a graded morphism of $\mathcal{O}_S$-algebras such that $\mathbf{P}(\xi)\varphi_1=\varphi_2$ on $\varphi_1^{-1}(\mathbf{G}(\xi))$, there exists $s\in\mathrm{Hom}_{\mathcal{O}_X}(\mathcal{L}_2,\mathcal{L}_1)$ such that
\begin{equation}\label{square}\xymatrix{(f^*\mathcal{A}_2)_1\ar[d]_{\psi_2}\ar[rr]^{f^*\xi}&&(f^*\mathcal{A}_1)_1\ar[d]^{\psi_1}\\\mathcal{L}_2\ar[rr]_{s}&&\mathcal{L}_1}\end{equation}
commutes. 

Reciprocally, let $s\in\mathrm{Hom}_{\mathcal{O}_X}(\mathcal{L}_2,\mathcal{L}_1)$ be such that the above diagram commutes.

\begin{lemma}\label{D(s)}
We have that $\mathbf{D}(s):=X\smallsetminus\mathbf{V}(s)=\varphi_1^{-1}(\mathbf{G}(\xi))$.
\end{lemma}
\begin{proof}
On $\mathbf{D}(s)$, $s$ gives an isomorphism between $\mathcal{L}_2|_{\mathbf{D}(s)}$ and $\mathcal{L}_1|_{\mathbf{D}(s)}$. Therefore, $\psi_2$ and $\psi_1(f^*\xi)$ induce the same morphisms from $\mathbf{D}(s)$ to $Y_2$, which are simply $\varphi_2$ and $\mathbf{P}(\xi)\varphi_1$, respectively. In particular, $\mathbf{P}(\xi)\varphi_1$ is defined on $\mathbf{D}(s)$, and so $\mathbf{D}(s)\subseteq\varphi_1^{-1}(\mathbf{G}(\xi))$.  

For the other inclusion, we have that $\mathbf{P}(\xi)\varphi_1$ and $\varphi_2$ are two morphisms on $\varphi_1^{-1}(\mathbf{G}(\xi))$ that coincide on the open set $\mathbf{D}(s)$. If $\mathbf{D}(s)=\varnothing$, then $s=0$ and the image of $f^*\xi$ is contained in $\ker\psi_1$. However this implies that $\varphi_1^{-1}(\mathbf{G}(\xi))=\varnothing$. If $\mathbf{D}(s)\neq\varnothing$, then since $Y_2$ is separated over $S$, these morphisms must be equal on the larger open set. By our previous analysis, there exists $s'\in\mathrm{Hom}_{\mathcal{O}_X}(\mathcal{L}_2,\mathcal{L}_1)$ that makes diagram \eqref{square} commute. However, $s$ and $s'$ coincide on $\mathbf{D}(s)$, and since $X$ is integral and noetherian, these sections must be equal. Moreover, also by our previous analysis, $s$ is an isomorphism over $\varphi_1^{-1}(\mathbf{G}(\xi))$, and so $\varphi_1^{-1}(\mathbf{G}(\xi))\subseteq\mathbf{D}(s)$.
\end{proof}

This lemma implies then that on $\varphi_1^{-1}(\mathbf{G}(\xi))$, $s$ is an isomorphism between $\mathcal{L}_2|_{\varphi_1^{-1}(\mathbf{G}(\xi))}$ and $\mathcal{L}_1|_{\varphi_1^{-1}(\mathbf{G}(\xi))}$, and therefore $\mathbf{P}(\xi)\varphi_1=\varphi_2$ on $\varphi_1^{-1}(\mathbf{G}(\xi))$. This therefore completes the proof of Theorem \ref{general theorem}. \qed\\

\section{Galois projections with fixed group}\label{galois}

In this section we will now study the situation given in the introduction. Let $X$ be a projective variety of dimension $n$ and let $\varphi:X\hookrightarrow\mathbb{P}^N$ be an embedding given by a linear system $V\leq H^0(X,\mathcal{L})$ for some invertible sheaf $\mathcal{L}$. Let $W\in\mathbf{G}^{\mathrm{aut}}_{X,\varphi}$, let $G:=\mathrm{Gal}(W)$ be its Galois group, and fix a quotient map $\pi_G:X\to\mathbb{P}^n$ by $G$.

\begin{lemma}\label{linear projection}
There is a 1-1 correspondence between elements of $\mathbf{G}_{X,\varphi}^{\mathrm{aut}}$ with Galois group equal to $G$ (not just isomorphic) and the set of all linear projections $\ell:\mathbb{P}^N\dashrightarrow\mathbb{P}^n$ such that the following diagram is commutative (where defined):
\[\xymatrix{X\ar[dr]_{\pi_G}\ar[r]^\varphi&\mathbb{P}^N\ar@{-->}[d]^\ell\\&\mathbb{P}^n}\]
\end{lemma}
\begin{proof}
Clearly if $\ell$ is a linear projection that makes the above diagram commute, then the base locus of $\ell$ is a Galois subspace for $\varphi$ with Galois group $G$. Reciprocally, if $W\in\mathbf{G}_{X,\varphi}^{\mathrm{aut}}$ has Galois group $G$, then $\pi_W\varphi$ can be extended to a $G$-invariant morphism $f:X\to\mathbb{P}^n$. By the uniqueness of the quotient map, there exists a unique automorphism $\sigma\in\mathrm{Aut}(\mathbb{P}^n)$ such that $\sigma f=\pi_G$. In particular, $\ell:=\sigma\pi_W$ is a linear projection that makes the above diagram commute.
\end{proof}

Since $\varphi$ is given by a linear system $V\leq H^0(X,\mathcal{L})$ and $\pi_G$ is given by the linear system $H^0(X,\pi_G^*\mathcal{O}(1))^G$ for the invertible sheaf $\pi_G^*\mathcal{O}(1)$, by Theorem \ref{general theorem}, we have that a linear projection $\ell:\mathbb{P}^N\dashrightarrow\mathbb{P}^n$ induced by a $k$-linear map $\xi:k[x_1,\ldots,x_n]_1\to k[y_1,\ldots,y_N]_1$ makes the diagram of the previous lemma commute if and only if there exists $s\in\mathrm{Hom}_{\mathcal{O}_X}(\mathcal{L}_2,\mathcal{L}_1)$ such that the diagram
\begin{equation}\label{diagram with s}\xymatrix{\mathcal{O}_X[x_0,\ldots,x_n]_1\ar[d]_{\pi_G^*}\ar[r]^\xi&\mathcal{O}_X[y_0,\ldots,y_N]_1\ar[d]^{\varphi^*}\\\pi_G^*\mathcal{O}(1)\ar[r]_s&\mathcal{L}}\end{equation}
commutes. However, the images of these vertical arrows are simply 
\[H^0(X,\pi_G^*\mathcal{O}(1))^G\otimes_k\mathcal{O}_X\hspace{0.5cm}\text{and}\hspace{0.5cm} V\otimes_k\mathcal{O}_X,\] 
respectively, and so what is needed is the existence of $s\in\mathrm{Hom}_{\mathcal{O}_X}(\pi_G^*\mathcal{O}(1)),\mathcal{L})$ such that 
\[s(H^0(X,\pi_G^*\mathcal{O}(1))^G\otimes_k\mathcal{O}_X)\subseteq V\otimes_k\mathcal{O}_X.\]
Now, using the well-known fact that $\mathrm{Hom}_{\mathcal{O}_X}(\mathcal{L}_1,\mathcal{L}_2)\simeq H^0(X,\mathcal{L}_2\otimes\mathcal{L}^{-1})$ for two line bundles $\mathcal{L}_1$ and $\mathcal{L}_2$, as well as the existence of a natural multiplication map
\[H^0(X,\mathcal{L}_1)\otimes H^0(X,\mathcal{L}_2)\to H^0(X,\mathcal{L}_1\otimes\mathcal{L}_2),\]
the previous condition is equivalent to the existence of $s\in H^0(X,\mathcal{L}\otimes\pi_G^*\mathcal{O}(-1))$ such that 
\begin{equation}\label{inclusion} s\otimes H^0(X,\pi_G^*\mathcal{O}(1))^G\subseteq V.\end{equation}
It is clear that such an $s$ that also makes diagram \ref{diagram with s} commute must be unique. 

Reciprocally, if $s\in H^0(X,\mathcal{L}\otimes\pi_G\mathcal{O}(-1))$ is such that (\ref{inclusion}) holds, then for each $0\leq j\leq n$, there exist $a_{ij}\in k$ such that
\[s\otimes\pi_G^*x_j=\sum_{i=0}^Na_{ij}\varphi^*y_i.\]
Therefore the matrix $(a_{ij})_{i,j}$ gives a $k$-linear map 
\[\xi:k[x_0,\ldots,x_n]_1\to k[y_0,\ldots,y_N]_1\]
which in turn gives us a linear map $\ell:=\mathbf{P}(\xi):\mathbb{P}^N\dashrightarrow\mathbb{P}^n$ such that $\ell\varphi=\pi_G$, when defined. By using Lemma \ref{linear projection}, we have therefore proven the following proposition:

\begin{proposition}\label{Galoisgroup}
There is a natural bijective correspondence between the set of Galois subspaces $W\in\mathbf{G}_{X,\varphi}^\mathrm{aut}$ with Galois group $G$ and (the closed points of)
\[\mathbb{P}_{G,\varphi}:=\mathbb{P}\{s\in H^0(X,\mathcal{L}\otimes\pi_G^*\mathcal{O}(-1)):s\otimes H^0(X,\pi_G^*\mathcal{O}(1))^G\subseteq V\}.\]
In particular, if $\varphi$ is given by a complete linear system, then there is a bijective correspondence between the set of Galois subspaces in $\mathbf{G}_{X,\varphi}^\mathrm{aut}$ with Galois group $G$ and the space $\mathbb{P}H^0(X,\mathcal{L}\otimes\pi_G^*\mathcal{O}(-1))$.
\end{proposition}

We note that the projectivization of the global sections in the above proposition is simply added since two $k$-linear maps 
\[\xi_1,\xi_2:k[x_0,\ldots,x_n]_1\to k[y_0,\ldots,y_n]_1\] 
give the same rational linear map between the associated projective spaces if and only if they differ by multiplication by a non-zero constant.

\begin{rem}\label{explicit} We note that given an element $[s]\in\mathbb{P}_{G,\varphi}$, multiplication by $s$ gives a $k$-linear map $L_s:H^0(X,\pi_G^*\mathcal{O}(1))^G\to V$ and the Galois subspace associated to $[s]$ is exactly 
\[\ker L_s^\vee=\{\ell\in V^\vee:s\otimes H^0(X,\pi_G^*\mathcal{O}(1))^G\subseteq\ker\ell\}.\]
\end{rem}

\begin{lemma}\label{embedding}
The natural map $\mathbb{P}_{G,\varphi}\to\mathbb{G}(N-n-1,N)$ is an embedding.
\end{lemma}
\begin{proof}
This is a simple exercise of writing everything out, but we will give a sketch of the proof anyways. Write 
\[W:=\{s\in H^0(X,\mathcal{L}\otimes\pi_G^*\mathcal{O}(-1)):s\otimes H^0(X,\pi_G^*\mathcal{O}(1))^G\subseteq V\},\] 
and consider the linear transformation
\[\chi:W\to\mathrm{Hom(H^0(X,\pi_G^*\mathcal{O}(1))^G,V)}\]
\[s\mapsto\left(\chi_s:t\mapsto s\otimes t\right).\]
We need to analyze the differential of the map $W\to G(n+1,V)$ at a point $p\neq0$, where $s\mapsto \mathrm{Im}(\chi_s)$. By choosing bases for all the vector spaces involved, we can write $\chi_s=(\ell_{ij}(s))_{ij}$ where $\chi_s$ is an $(N+1)\times(n+1)$ matrix and each $\ell_{ij}$ is linear. Since $\chi_s$ is a monomorphism, we can assume, without loss of generality, that
\[\chi_s=\left(\begin{array}{c}N(s)\\M(s)\end{array}\right)\]
where $N(s)$ is invertible in a neighborhood of $p$. Now in affine coordinates, in a neighborhood of $p$ we have that $\chi$ can be identified with the map $s\mapsto M(s)N(s)^{-1}$. The differential of $\chi$ at $p$ is therefore $x\mapsto M(x)-M(p)N(p)^{-1}N(x)$. Since $\mathrm{Im}(\chi_s)=\mathrm{Im}(\chi_t)$ if and only if $s$ is a multiple of $t$, we obtain that $\ker(M(x)-M(p)N(p)^{-1}N(x))=\langle p\rangle$, and therefore the map
\[\mathbb{P}W\to G(n+1,V)\simeq\mathbb{G}(N-n-1,N)\]
is an embedding.
\end{proof}

\section{Families of Galois subspaces}\label{families} 

We now wish to relativize this construction in order to obtain families of Galois subspaces for $\varphi$. Following \cite[Chapter 4]{Kleiman2} (although with different notation), let $\mathbf{Div}_{\geq0}(X)$ denote the open subscheme of $\mathbf{Hilb}_{X/S}$ that parametrizes effective Cartier divisors on $X$, and let
\[\alpha:\mathbf{Div}_{\geq0}(X)\to\mathbf{Pic}(X)\]
\[D\mapsto\mathcal{O}_X(D)\]
denote the Abel map. We also have the morphism
\[\beta_G:\mathrm{Aut}(X)\to\mathbf{Pic}(X)\]
\[\theta\mapsto \mathcal{L}\otimes\theta^*\pi_G^*\mathcal{O}(-1).\]

Let $G\leq\mathrm{Aut}(X)$ be the Galois group of a Galois subspace $W\in\mathbf{G}_{X,\varphi}^\mathrm{aut}$, and consider the fiber product $\mathrm{Aut}^0(X)\times_{\mathbf{Pic}(X)}\mathbf{Div}_{\geq0}(X)$.
Here the map $\mathrm{Aut}^0(X)\to\mathbf{Pic}(X)$ is just the restriction of $\beta_G$. By \cite[Prop. 8.2.7]{Neron}, there exists a coherent sheaf $\mathcal{F}_G$ on $\mathrm{Aut}^0(X)$ such that
\[\mathrm{Aut}^0(X)\times_{\mathbf{Pic}(X)}\mathbf{Div}_{\geq0}(X)=\mathbf{P}(\mathcal{F}_G),\]
where $\mathbf{P}(\mathcal{F}_G):=\mathrm{Proj}(\mathrm{Sym}(\mathcal{F}_G))$.

We have the first projection $p_1:\mathbf{P}(\mathcal{F}_G)\to\mathrm{Aut}^0(X)$, and for any $\theta\in\mathrm{Aut}^0(X)$, the fiber is set-theoretically equal to
\begin{eqnarray}\nonumber p_1^{-1}(\theta)&=&\{D\in\mathbf{Div}_{\geq0}(X):\mathcal{O}_X(D)\simeq\mathcal{L}\otimes\theta^*\pi_G^*\mathcal{O}(-1)\}\\
\nonumber&=&|\mathcal{L}\otimes\theta^*\pi_G^*\mathcal{O}(-1)|\\
\nonumber&\simeq&\mathbb{P}H^0(X,\mathcal{L}\otimes\pi_{\theta G\theta^{-1}}^*\mathcal{O}(-1)).\end{eqnarray}
since $\pi_G\theta$ is just the quotient map of the group $\theta G\theta^{-1}$. This means that $\mathbf{P}(\mathcal{F}_G)$ is a scheme that parametrizes, over the base $\mathrm{Aut}^0(X)$, Galois subspaces for $\varphi$ whose Galois group is conjugate to $G$ by an element of $\mathrm{Aut}^0(X)$.

Now clearly there are fibers of $p_1$ that can coincide. To fix this, we see that the normalizer $N(G)$ of $G$ in $\mathrm{Aut}^0(G)$ acts by multiplication on the right on the first coordinate of $\mathbf{P}(\mathcal{F}_G)$ and the first projection is equivariant with respect to this action. Therefore we obtain a fibration
\[\mathbf{P}(\mathcal{F}_G)/N(G)\to \mathrm{Aut}^0(X)/N(G)\]
where each fiber corresponds to Galois subspaces for a unique conjugate of $G$. In particular, the natural map 
\[\mathbf{P}(\mathcal{F}_G)/N(G)\to\mathbf{G}(N-n-1,N)\]
is injective.

Since $N(G)$ acts freely on $\mathbf{P}(\mathcal{F}_G)$ and $\mathrm{Aut}^0(X)$ (and doesn't affect the fibers of the relative proj), we have that there exists a coherent sheaf $\mathcal{E}_G$ on $\mathrm{Aut}^0(X)/N(G)$ such that $\mathbf{P}(\mathcal{F}_G)/N(G)=\mathbf{P}(\mathcal{E}_G)$. We notice that the isomorphism class of $\mathbf{P}(\mathcal{E}_G))$ only depends on the orbit of $G$ under conjugation by $\mathrm{Aut}^0(X)$. 

\begin{rem}\label{not irreducible}
We note that $\mathbf{P}(\mathcal{E}_G)$ is not necessarily irreducible. Indeed, consider the case where $X$ is a general elliptic curve and $\varphi$ is the embedding associated to the complete linear system of an ample divisor $D$ of degree $N+1$. By \cite[Theorem 1.2]{AuffRah}, if $N\geq3$ is odd, then there exist disjoint Galois subspaces, and there are finitely many of them. On the other hand, if $G$ is a Galois group for a disjoint Galois subspace, then $N(G)$ is finite, and $\mathrm{Aut}^0(X)/N(G)=X/N(G)$ is infinite. The point is that there are only finitely many conjugates $\theta G\theta^{-1}$ of $G$ such that $\dim H^0(X,\mathcal{O}_X(D)\otimes\pi_{\theta G\theta^{-1}}^*\mathcal{O}(-1))>0$.
\end{rem}

\begin{rem}\label{stab}
Note that if $D\in\mathbf{Div}_{\geq0}(X)$, then if $\theta_0\in p_2^{-1}(D)$, we get that
\[p_2^{-1}(D)=\theta_0\mathrm{Stab}_{\mathrm{Aut}^0(X)}(\pi_G^*\mathcal{O}(1)).\]
In the same vein, we observe that if $G$ is such that $\pi_G^*\mathcal{O}(1)\simeq\mathcal{L}$, then
\[\mathbf{P}(\mathcal{E}_G)\simeq\mathrm{Stab}_{\mathrm{Aut}^0(X)}(\pi_G^*\mathcal{O}(1))/N(G).\]
\end{rem}

Now in order to finish the proof of Theorem \ref{main}, we only need the following:

\begin{proposition}
The natural injective map $\mathbf{P}(\mathcal{E}_G)\to\mathbf{G}(N-n-1,N)$ is an embedding.
\end{proposition}
\begin{proof}
First, we see that it is a morphism. Indeed, identifying $\mathbb{P}V$ as a subspace of $|\mathcal{L}|$, first take the morphism
\[\mathbf{P}(\mathcal{F}_G)\to \mathbb{G}(n,\mathbb{P}V)\]
\[(\theta,D)\mapsto D+\theta^*\pi_G^*|\mathcal{O}(1)|\]
and then compose it with the isomorphism
\[\mathbb{G}(n,\mathbb{P}V)\simeq G(n+1,V)\to G(N-n,V^\vee)\simeq \mathbb{G}(N-n-1,\mathbb{P}V^\vee)\]
\[W\mapsto\{\ell\in V^\vee:W\subseteq\ker\ell\}.\]
By Remark \ref{explicit}, after correct identification, this composition is just the map $\mathbf{P}(\mathcal{F}_G)\to\mathbb{G}(N-n-1,N)$. Since it is $N(G)$-invariant, we have that the map from $\mathbf{P}(\mathcal{E}_G)\to\mathbb{G}(N-n-1,N)$ is a morphism. 

To see that it is an embedding, it is necessary to do a similar calculation as in the proof of Lemma \ref{embedding}, but this time taking into account the variation of the automorphism. We leave the details to the reader.\end{proof}

\begin{rem}\label{general}
Note that nowhere did we use that $\varphi$ is an embedding in order to describe $\mathbf{G}_{X,\varphi}^{\mathrm{aut}}$. Indeed, in every proof we only used the existence of the linear system $V$, not the fact that it separates points and tangent vectors. Therefore, one could also define $\mathbf{G}_{X,\varphi}^{\mathrm{aut}}$ in exactly the same way for \textit{any} morphism $\varphi:X\to\mathbb{P}^N$.
\end{rem}

\section{Intermediate projections}\label{intermediate}

Consider the same situation as before. Note that if $G\leq\mathrm{Aut}(X)$ is a finite subgroup such that $X/G\simeq\mathbb{P}^n$, then up until now we have been using principally two linear systems, namely $H^0(X,\pi_G^*\mathcal{O}(1))^G$ and a sublinear system $V\leq H^0(X,\mathcal{L})$. These give the maps $\varphi$ and $\pi_G$, respectively. However, there is another linear system that also plays an interesting role: $Q_G:=H^0(X,\pi_G^*\mathcal{O}(1))$. This linear system, along with any Galois subspace $W\in\mathbb{G}(N-n-1,N)$ with Galois group $G$, gives us a commutative diagram
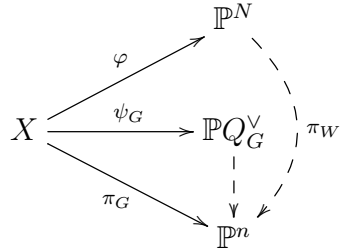
\[\xymatrix{&&\mathbb{P}^N\ar@{-->}@/^2pc/[dd]^{\pi_W}\\X\ar[drr]_{\pi_G}\ar[urr]^{\varphi}\ar[rr]^{\psi_G}&&\mathbb{P}Q_G^\vee\ar@{-->}[d]\\&&\mathbb{P}^n}\]
where $\psi_G$ is the map defined by $Q_G$ and the vertical rational map is given by the inclusion $Q_G^G\hookrightarrow Q_G$. Now $\mathrm{Stab}_{\mathrm{Aut}^0(X)}(\pi_G^*\mathcal{O}(1))$ acts linearly on $Q_G$, and therefore on $\mathbb{P}Q_G^\vee$. By Remark \ref{stab} (and noting by Remark \ref{general} that the following set is well-defined), we get that $\mathbf{G}_{X,\psi_G}\simeq \mathrm{Stab}_{\mathrm{Aut}^0(X)}(\pi_G^*\mathcal{O}(1))/N(G)$, and \textit{every} Galois subspace here is disjoint from $\psi_G(X)$. Note that for every $\theta\in \mathrm{Stab}_{\mathrm{Aut}^0(X)}(\pi_G^*\mathcal{O}(1))$, $\pi_G^*\mathcal{O}(1)\otimes\pi_{\theta G\theta^{-1}}^*\mathcal{O}(-1)\simeq\mathcal{O}_X$, and so $\mathbb{P}_{\theta G\theta^{-1},\psi_G}$ is a point. In particular, we get the following result:

\begin{proposition}
For every $G\leq\mathrm{Aut}(X)$ such that $X/G\simeq\mathbb{P}^n$, we get a surjective morphism $\mathbf{P}(\mathcal{E}_G)\to \mathrm{Stab}_{\mathrm{Aut}^0(X)}(\pi_G^*\mathcal{O}(1))/N(G)\hookrightarrow\mathbf{G}_{X,\psi_G}^\varnothing$ whose fiber over the class of $\theta$ is $\mathbb{P}_{\theta G\theta^{-1},\varphi}$.
\end{proposition}

\begin{rem}
We note that this is exactly the map $\Phi_{n,m}$ defined in \cite[Section 6]{AR} that was used in order to understand the families of Galois subspaces.
\end{rem}

Now, by the previous results, for every $W\in\mathbf{G}_{X,\varphi}^{\mathrm{aut}}$ with Galois group $G$, we get an element $[s_W]\in\mathbb{P}H^0(X,\mathcal{L}\otimes\pi_G^*\mathcal{O}(-1))$ such that $s_W\otimes H^0(X,\pi_G^*\mathcal{O}(1))^G\subseteq V$. Note that if $V$ is the complete linear system, then we get that there exists a rational linear map $\mathbb{P}^N\dashrightarrow\mathbb{P}Q_G^\vee$ given by $s_W$ (and the inclusion $s_W\otimes H^0(X,\mathcal{L}\otimes\pi_G^*\mathcal{O}(-1))\subseteq H^0(X,\mathcal{L})=V$) that makes the above diagram commute.

\section{Final examples/questions}\label{questions}

In this section we will show two examples that lead us to natural questions for future research.

\subsection{Example 1} Note that if $\mathrm{Aut}^0(X)$ is not compact, then the topology of $\mathbf{G}_{X,\varphi}^{\mathrm{aut}}$ can be quite interesting. For example, if $X=\mathbb{P}^1$ and $\varphi$ is the Veronese embedding of degree $d$, then for any $G\leq\mathrm{PGL}(2,k)=\mathrm{Aut}(\mathbb{P}^1)$ with $|G|\leq d$, we get a $\mathbb{P}^{d-|G|}$-bundle $\mathbf{P}(\mathcal{E}_G)$ over $\mathrm{PGL}(2,k)/N(G)$ embedded in $\mathbf{G}(d-2,d)$. Now since the base $\mathrm{PGL}(2,k)/N(G)$ is not compact, the image of $\mathbf{P}(\mathcal{E}_G)$ in $\mathbf{G}(d-2,d)$ is not compact. We observe, for example, that for $d=2$, 
\[\mathbf{G}_{\mathbb{P}^1,\varphi}^{\mathrm{aut}}=\mathbf{G}_{\mathbb{P}^1,\varphi}=\mathbf{G}(0,2)=\mathbb{P}^2=\mathbf{P}(\mathcal{E}_{\{1\}})\sqcup\mathbf{P}(\mathcal{E}_{\mathbb{Z}/2\mathbb{Z}}),\]
where $\mathbb{Z}/2\mathbb{Z}$ is the group generated by $[x:y]\mapsto[-x:y]$. Indeed, $\mathbf{P}(\mathcal{E}_{\{1\}})$ is just the degree 2 rational normal curve, and $\mathbf{P}(\mathcal{E}_{\mathbb{Z}/2\mathbb{Z}})$ is its complement. In particular, $\mathbf{P}(\mathcal{E}_{\{1\}})$ lies in the closure of $\mathbf{P}(\mathcal{E}_{\mathbb{Z}/2\mathbb{Z}})$.  It would be interesting to understand the topology of the families of Galois subspaces discovered in \cite[Theorem 6.3]{AR}. More generally, we ask:\\

\noindent\textbf{Question 1:} If $\mathrm{Aut}^0(X)$ is not proper over $k$, how can we describe the topology (i.e. irreducible components, etc.) of $\mathbf{G}_{X,\varphi}^{\mathrm{aut}}$?\\

\subsection{Example 2} If $A$ is an abelian variety of dimension $n$ and $\phi:A\hookrightarrow\mathbb{P}^N$ is a non-degenerate embedding, then if $\mathbf{G}_{X,\varphi}^{\mathrm{aut}}\neq\varnothing$, by \cite[Theorem 1.2]{Auffarth} $A$ must be isogenous to the self product of an elliptic curve. Moreover, by the main result of \cite{ALA}, if $G\leq\mathrm{Aut}(A)$ is a finite subgroup such that $A/G\simeq\mathbb{P}^{n}$ and $T\unlhd G$ is the subgroup of $G$ that consists of translations, then $A/T$ is \textit{isomorphic} to the self-product of an elliptic curve. It can be proven that $G/T$ acts on $A/T$ fixing the origin (although this is far from being trivial), and again by the main result of \cite{ALA}, $G/T$ must be isomorphic (i.e. conjugate) to one of six possible groups. In what follows we will stick with the self-product of an elliptic curve $E$. We have that 
\[\mathrm{Aut}(E^n)/\mathrm{Aut}^0(E^n)\simeq \mathrm{GL}(n,\mathrm{End}(E))\]
If $G\leq \mathrm{Aut}(E^n)$ fixes the origin and is such that $E^n/G\simeq\mathbb{P}^n$, and $\varphi:E^n\hookrightarrow\mathbb{P}^N$ is an embedding, then by Theorem \ref{main} we get a projective bundle $\mathbf{P}(\mathcal{E}_G)$ over a quotient of $E^n$ that represents all Galois subspaces for $\varphi$ with Galois group a conjugate of $G$ by an element of $\mathrm{Aut}^0(E^n)=E^n$. However, we could also conjugate $G$ by an element $\sigma\in\mathrm{GL}(n,\mathrm{End}(E))$ and obtain a different projective bundle $\mathbf{P}(\mathcal{E}_{\sigma G\sigma^{-1}})$. After having worked out several examples, we would like to conjecture the following:\\

\noindent\textbf{Conjecture:} If $E$ is an elliptic curve, $\varphi:E^n\hookrightarrow\mathbb{P}^N$ is a non-degenerate embedding, and $G\leq\mathrm{Aut}(E^n)$ fixes the origin and is such that $E^n/G\simeq\mathbb{P}^n$, then the set
\[\{\sigma\in\mathrm{GL}(n,\mathrm{End}(E)):\mathbf{P}(\mathcal{E}_{\sigma G\sigma^{-1}})\neq\varnothing\}\]
is finite.\\

In particular, an affirmative answer to this conjecture would imply that if $A$ is any abelian variety and $\varphi:A\hookrightarrow\mathbb{P}^N$ is a non-degenerate embedding, then $\mathbf{G}_{A,\varphi}^{\mathrm{aut}}$ is a subvariety of $\mathbb{G}(N-n-1,N)$ (i.e. it has finitely many components). This leads us to our last question:\\

\noindent\textbf{Question 2:} Does there exist a projective variety $X$ and a non-degenerate embedding $\varphi:X\hookrightarrow\mathbb{P}^N$ such that $\mathbf{G}_{X,\varphi}^{\mathrm{aut}}$ is not of finite type?

\end{document}